\documentclass{amsart}
\usepackage{fullpage}
\usepackage{amsfonts}

\usepackage[usenames,dvipsnames]{color}
\usepackage{comment}
\usepackage{amsthm}

\newcommand{\kommentar}[1]{}
\newcommand{\C}{\mathbb{C}}

\newcommand{\R}{\mathbb{R}}
\newcommand{\Q}{\mathbb{Q}}
\newcommand{\N}{\mathbb{N}}
\newcommand{\Z}{\mathbb{Z}}

\newcommand{\HH}{\mathbb{H}}

\newcommand{\tr}{\mathrm{tr\;}}

\newtheorem{thm}{Theorem}
\newtheorem{defn}[thm]{Definition}
\newtheorem{prop}[thm]{Proposition}

\newtheorem{lem}[thm]{Lemma}
\newtheorem{conj}[thm]{Conjecture}

\title{Counting Salem numbers of arithmetic hyperbolic $3$--orbifolds}

\author{Mikhail Belolipetsky}
\address{IMPA, Estrada Dona Castorina, 110, 22460-320 Rio de Janeiro, Brazil}
\email{mbel@impa.br}

\author{Matilde Lal\'in}
\address{Universite de Montr\'eal, Pavillon Andr\'e-Aisenstadt, D\'ept. de math\'ematiques et de statistique, CP 6128, succ. Centre-ville Montr\'eal, Qu\'ebec, H3C 3J7, Canada}
\email{mlalin@dms.umontreal.ca}

\author{Plinio G. P. Murillo}
\address{Universidade Federal Fluminense, Instituto de Matem\'atica e Estat\'istica\\
Rua Prof. Marcos Waldemar de Freitas Reis, S/n, Bloco H, Campus do Gragoat\~a, 24210-201. Niter\'oi - RJ, Brazil}
\email{pliniom@id.uff.br}

\author{Lola Thompson}
\address{Mathematisch Instituut, Universiteit Utrecht, Budapestlaan 6, 3584 CD Utrecht, The Netherlands}
\email{l.thompson@uu.nl}


\begin{document}

\begin{abstract}
It is known that the lengths of closed geodesics of an arithmetic hyperbolic orbifold are related to Salem numbers. We initiate a quantitative study of this phenomenon. We show that any non-compact arithmetic $3$-dimensional orbifold defines $c Q^{1/2} + O(Q^{1/4})$ square-rootable Salem numbers of degree $4$ which are less than or equal to $Q$. This quantity can be compared to the total number of such Salem numbers, which is shown to be asymptotic to $\frac{4}{3}Q^{3/2}+O(Q)$. Assuming the gap conjecture of Marklof, we can extend these results to compact arithmetic $3$-orbifolds.  As an application, we obtain lower bounds for the strong exponential growth of mean multiplicities in the geodesic spectrum of non-compact even dimensional arithmetic orbifolds. Previously, such lower bounds had only been obtained in dimensions $2$ and $3$. 
\end{abstract}

\maketitle

\section{Introduction}

A \emph{Salem number} is a real algebraic integer $\lambda > 1$ such that all of its Galois conjugates except $\lambda^{-1}$ have absolute value equal to $1$. Salem numbers appear in many areas of mathematics including algebra, geometry, dynamical systems, and number theory. They are closely related to the celebrated Lehmer's problem about the smallest Mahler measure of a non-cyclotomic polynomial. We refer to \cite{Smyth-survey} for a survey of research on Salem numbers.   

It has been known for some time that the exponential lengths of the closed geodesics of an arithmetic hyperbolic $n$-dimensional manifold or orbifold are given by Salem numbers. For $n = 2$ and $3$ this relation is described in the book by C.~Maclachlan and A.~Reid \cite[Chapter 12]{MR-book}. More recently, it was elaborated upon and generalized to higher dimensions by V.~Emery,  J.~Ratcliffe, and S.~Tschantz in \cite{EmeryRatcliffeTschantz}. In particular, their Theorem~1.1 implies that, for a non-compact arithmetic  hyperbolic $n$-orbifold $\mathcal{O}$, a closed geodesic of length $\ell$ corresponds to a Salem number $\lambda = e^\ell$ if the dimension $n$ is even, and to a so called square-rootable Salem number $\lambda = e^{2\ell}$ if $n$ is odd. The degrees of these Salem numbers satisfy $\deg(\lambda) \le n+1$. A natural question arises: 
\emph{
What proportion of Salem numbers of a given degree are associated to a fixed orbifold $\mathcal{O}$?
}

To this end, let us recall some results about the distribution of algebraic integers. This field has a long history, so we will mention only the more recent results which are relevant to our work. In a beautiful paper \cite{Thurston-entropy}, W.~Thurston, motivated by the study of entropy of one-dimensional dynamical systems, encountered limiting distributions of conjugates of Perron numbers, a class which includes Salem numbers as a subset. His experiments led to a set of interesting problems and conjectures, some of which were successfully resolved by F.~Calegari and Z.~Huang in \cite{CalegariHuang}. Later on, some ideas from their approach helped F.~G\"otze and A.~Gusakova  to compute the asymptotic growth of Salem numbers in \cite{GoetzeGusakova}. The precise form of their result is given in Theorem~\ref{thm:GG}. It is remarkable that this result was established only very recently, as it allows us to play the asymptotic formula against the distribution of closed geodesics of an arithmetic $n$-orbifold. We also come up with a related question about the distribution of square-rootable Salem numbers. We were able to answer these questions in the first non-trivial case, when the degree of the Salem numbers is $4$ and the corresponding dimension of the arithmetic orbifolds is $3$. For higher even dimensions, the interplay between counting Salem numbers and the prime geodesic theorems (\cite{Margulis}, \cite{GanWar}) allows us to prove lower bounds for the strong exponential growth of mean multiplicities in the geodesic spectrum.


The first result of this paper is the following theorem:

\newpage

\begin{thm}\label{thm1}\
\begin{itemize}
\item[A.] Let $\mathcal{O}_D$ be a non-compact arithmetic hyperbolic $3$-orbifold associated to a Bianchi group $\Gamma_D = \mathrm{PSL}(2, \mathfrak{o}_K)$, where $\mathfrak{o}_K$ is the ring of integers of an imaginary quadratic number field $K=\Q(\sqrt{-D})$, and $D$ is a square-free positive integer. Then $\mathcal{O}_D$ generates 
\[c Q^{1/2} + O(Q^{1/4})\]
square-rootable Salem numbers of degree $4$ which are less or equal to $Q$, where $c = \frac{\pi}{4\sqrt{D}}$ if $D \equiv 1, 2\ \mathrm{mod}\ 4$ and $c = \frac{\pi}{2\sqrt{D}}$ if $D \equiv 3\ \mathrm{mod}\ 4$.
\item[B.] The number of Salem numbers of degree $4$ that are square-rootable over $\Q$ and less than or equal to $Q$ is 
\[\frac{4}{3}Q^{3/2}+O(Q).\]
\end{itemize}
\end{thm}

This theorem together with a special case of the theorem of G\"otze and Gusakova implies that, in the logarithmic scale, a given $3$-orbifold $\mathcal{O}_D$ generates asymptotically $1/4$ of all Salem numbers of degree $4$ and asymptotically $1/3$ of the square-rootable Salem numbers of degree $4$.

The proof of the first part of the theorem uses the work of J.~Marklof on multiplicities in length spectra of arithmetic hyperbolic $3$--orbifolds \cite{Marklof}. For proving part~B we take advantage of some special properties of Salem numbers of degree 4.  Extending these results to higher degrees would require an extension of Marklof's length spectrum asymptotic to arithmetic orbifolds of dimension greater than $3$ and an analogue of the G\"otze--Gusakova theorem for square-rootable Salem numbers of higher degree. 

While we do not handle the higher-dimensional case in this paper, we are able to extend Theorem~\ref{thm1} to compact arithmetic $3$-orbifolds and associated Salem numbers of degree $4d$, $d\ge 1$, square-rootable over an intermediate field $L$ of degree $d$ --- see Theorem~\ref{thm2}. Finally, we show how to apply our methods to prove lower bounds for the strong exponential growth of mean multiplicities in the geodesic spectrum of non-compact even dimensional arithmetic orbifolds. Namely, we prove:

\begin{prop} \label{prop-mult}
Let $\mathcal{O}$ be a non-compact arithmetic hyperbolic orbifold of even dimension $n \ge 4$. Then the mean multiplicities in the length spectrum of $\mathcal{O}$ have a strong exponential growth rate of at least 
$$ \langle g(\ell) \rangle \sim c\frac{e^{(\frac{n}2-1)\ell}}{\ell}, \quad \ell\to\infty,$$
where $c$ is a positive constant. 
\end{prop}

\noindent Results of this nature were previously known only for arithmetic orbifolds of dimensions $2$ and $3$.

The paper is organized as follows. In Section~2 we recall definitions and some properties of Salem numbers and arithmetic groups. In Section~3 we prove Theorem~\ref{thm1}. In Section~4 we prove Theorem~\ref{thm2}. In Section~5 we consider other dimensions, prove Proposition~\ref{prop-mult}, and discuss some open problems.

\section{Preliminaries}\label{sec:preliminaries}

\subsection{Salem numbers} A \emph{Salem number} is a real algebraic integer $\lambda > 1$ such that all of its Galois conjugates have absolute value less than
or equal to $1$, and at least one of them has absolute value equal to $1$. Let $\lambda'$ denote a Galois conjugate of the Salem number $\lambda$ with $|\lambda'| = 1$. Since $\lambda'$ and its complex conjugate $\overline{\lambda'} = (\lambda')^{-1}$ are Galois conjugates we conclude that the minimal polynomial $p_\lambda$ of a Salem number $\lambda$ is self-reciprocal, i.e., $p_\lambda(x) = x^{\deg(p_\lambda)} p_\lambda(x^{-1})$. This means that its coefficients form a palindromic sequence. Moreover, the polynomial $p_\lambda$ is of even degree $n = 2(m + 1)$ because otherwise $p_\lambda(-1) = - p_\lambda(-1) = 0$, which contradicts its irreducibility. Thus, all Galois conjugates of a Salem number $\lambda$ (except for $\lambda^{-1}$) have absolute value $1$ and lie on the unit circle in the complex plane. We used this property as the first definition of Salem numbers in the introduction. For convenience in this paper we will also allow Salem numbers to have degree $2$ (where $\lambda$ has $\lambda^{-1}$ as its only conjugate). 
We do this so that our definition of Salem numbers aligns with that used in \cite{EmeryRatcliffeTschantz}.

The celebrated Lehmer's problem asks about the existence of a smallest Salem number $\lambda > 1$ and gives the conjectural candidate $\lambda = 1.176\ldots$ of degree $10$ found by D. H.~Lehmer in 1933 \cite{Lehmer}. We refer to \cite{Smyth-survey} and the references therein for more about Lehmer's problem. In this paper we will be interested in a somewhat opposite question about how quickly the number of Salem numbers grows when their values tend to infinity.

We denote by $\mathrm{Sal}_m$ the set of all Salem numbers of degree $2(m+1)$ and let 
\[\mathrm{Sal}_m(Q):= \{\lambda\in \mathrm{Sal}_m\,:\lambda \leq Q\}.\]

It is not hard to find examples of Salem numbers of any even degree. One of the basic facts proved by Salem is that if $\lambda$ is a Salem number of degree $n$, then so is $\lambda^k$ for all $k \in \N$ (see \cite[Lemma 2]{Smyth-survey}). This implies that the counting function $\# \mathrm{Sal}_m(Q)$ grows at least as fast as $c \log Q$. However, the actual growth of Salem numbers is much faster. A precise result for their asymptotic growth was recently proved by G\"otze and Gusakova:

\begin{thm}[\cite{GoetzeGusakova}, Theorem 1.1] \label{thm:GG} For any positive integer $m$ we have 
\[\# \mathrm{Sal}_m(Q)=\omega_m Q^{m+1} +O(Q^m),\]
as $Q\rightarrow \infty$, where
 \[\omega_m:=\frac{2^{m(m+1)}}{m+1} \prod_{k=0}^{m-1} \frac{k!^2}{(2k+1)!}.\]
\end{thm}
We remark for future reference that $\omega_1=2$. 


In this paper, we will encounter a special class of Salem numbers that are called \textit{square-rootable}.
\begin{defn}\label{defn-sr}
Let $\lambda$ be a Salem number, let $L$ be a subfield of $\Q(\lambda+\lambda^{-1})$, and let $p(x)$ be the minimal polynomial of $\lambda$ over $L$. We say that $\lambda$ is \emph{square-rootable} over $L$ via $\alpha$ if there exist a totally positive element $\alpha \in L$ and a monic palindromic polynomial $q(x)$, whose even degree coefficients are in $L$ and whose odd degree coefficients are in $\sqrt{\alpha}L$, such that $q(x)q(-x)=p(x^2)$. 
\end{defn}

The square-rootable Salem numbers were first defined by Emery, Ratclifffe, and Tschantz in \cite{EmeryRatcliffeTschantz}. They are of interest because they are associated to geodesic lengths of odd-dimensional arithmetic orbifolds.

\subsection{Arithmetic orbifolds}

Let us recall the definition of an arithmetic Kleinian group. Let $K$ be a number field with exactly one complex place, $\mathfrak{o}_K$ its ring of integers, and $A$ a quaternion algebra over $K$ ramified at all real places of $K$. Let $\mathfrak{D}$ be an $\mathfrak{o}_{K}$-order of $A$, and denote by $\mathfrak{D}^1$ its group of elements of norm $1$. 
Consider a $K$-embedding $\rho: A \hookrightarrow \mathrm{M}(2,\C)$ associated with the complex place of $K$. The group
\[ \Gamma_\mathfrak{D} = P\rho(\mathfrak{D}^1) < \mathrm{PSL}(2,\C), \]
where $P: \mathrm{SL}(2,\C) \to \mathrm{PSL}(2,\C)$ is the natural projection, is then a discrete finite covolume subgroup of $\mathrm{PSL}(2,\C)$. Following \cite{Marklof}, we will call $\Gamma_\mathfrak{D}$ an \emph{arithmetic quaternion group}.  A subgroup $\Gamma < \mathrm{PSL}(2,\C)$ which is commensurable with some such group $\Gamma_\mathfrak{D}$ is called an \emph{arithmetic Kleinian group}. The associated quotient space $\mathcal{O} = \HH^3/\Gamma$ is an \emph{arithmetic hyperbolic $3$-orbifold}.   

Arithmetic hyperbolic orbifolds can be compact or non-compact with cusps and finite volume. In dimension $3$ the non-compact orbifolds correspond to the arithmetic Kleinian groups which are commensurable with the Bianchi groups $\Gamma_D = \mathrm{PSL}(2, \mathfrak{o}_K)$, where $\mathfrak{o}_K$ is the ring of integers of an imaginary quadratic number field $K=\Q(\sqrt{-D})$ and $D$ is a square-free positive integer (cf. \cite[Theorem~8.2.3]{MR-book}).

An important subclass of arithmetic groups of hyperbolic isometries is defined by admissible quadratic forms. For this definition let $L$ be a totally real number field with ring of integers $\mathfrak{o}_L$, and let $f$ be a quadratic form of signature $(n,1)$ defined over $L$ such that, for every non-identity embedding $\sigma :L\to\R$, the form $f^\sigma$ is positive definite. The group $\Gamma = \mathrm{O}_0(f,\mathfrak{o}_L)$ of integral automorphisms of $f$ is a discrete subgroup of $\mathrm{O}_0(n,1)$, which is the full group of isometries of the hyperbolic $n$-space $\HH^n$ (the group $\mathrm{O}_0(n,1)$ is the subgroup of the orthogonal group $ \mathrm{O}(n,1)$ that preserves the upper cone in the vector model of $\HH^n$). Using reduction theory, one can show that such groups $\Gamma$ have finite covolume. The groups $\Gamma$ obtained in this way and subgroups of $\mathrm{Isom}(\HH^n)$ which are commensurable with them are called \emph{arithmetic subgroups of the simplest type}.

It is a well-known consequence of \cite{Weil1960} that every non-cocompact arithmetic $\Gamma$ in dimension $n \neq 7$ is commensurable with the group of units of a quadratic form over $\Q$. A more careful analysis implies that the same is true for $n = 7$ but we will not consider this case here. For the purpose of this paper we record a corollary that the Bianchi groups are arithmetic subgroups of the simplest type. Indeed, it is not hard to write down the corresponding quadratic forms for each $\Gamma_D$, which we leave as an exercise for the interested reader.

\section{Proof of Theorem \ref{thm1}}\label{sec2}


\subsection{Proof of part A}\label{sec:proof A}
In \cite{EmeryRatcliffeTschantz}, V.~Emery, J.~Ratcliffe and S.~Tschantz obtained the following result.
\begin{thm}[cf. \cite{EmeryRatcliffeTschantz}, Theorem 1.6]\label{thm:ERT}
Let $\Gamma \subseteq \mathrm{Isom}(\HH^n)$ be an arithmetic lattice, with $n$ odd, of the simplest type defined over a totally real number field $L$. Let $\gamma$ be a hyperbolic element of $\Gamma$, and let $\lambda=e^{2\ell(\gamma)}$. Then $\lambda$ is a Salem number which is square-rootable over $L$.  
\end{thm}

Recall from the previous section that in dimension $n=3$ the Bianchi groups correspond to arithmetic lattices of the simplest type defined over $L=\Q$. 

Now recall a result of Marklof on counting geodesic lengths in the spectrum of Bianchi orbifolds \cite{Marklof}. It is important for us that the lengths are counted without multiplicities, as it is well-known that multiplicities in the spectrum of arithmetic orbifolds can grow very rapidly. We have: 

\begin{thm}[cf. \cite{Marklof}, Theorem 4(b)]\label{thm:M}
Let $\Gamma_D= \mathrm{SL}(2, \mathfrak{o}_K)$, $K=\Q(\sqrt{-D})$, $D\in \Z_{>0}$ square-free. Then the number of distinct real lengths of closed geodesics less than or equal to $\ell$ in $\HH^3/\Gamma_D$ is given by 
\[\mathcal{N}_r(\ell)=\begin{cases}
                               \displaystyle{\frac{\pi}{4\sqrt{D}} e^\ell +O(e^{\ell/2})}, & D\equiv 1, 2\mod{4},\\\\
                                \displaystyle{\frac{\pi}{2\sqrt{D}} e^\ell +O(e^{\ell/2})}, & D\equiv 3\mod{4}.
                              \end{cases}\]
\end{thm}

Combining Theorems \ref{thm:ERT} and \ref{thm:M}, we obtain that a non-compact arithmetic hyperbolic 3-orbifold generates  
\[ \mathcal{N}(Q) = \mathcal{N}_r (\frac12 \log Q) \sim c Q^{1/2}\]
square-rootable Salem numbers $\lambda \leq Q$ of degree $\le 4$, where $c$ is the constant given by Theorem \ref{thm:M}. 

Going more carefully through the proof of Marklof's theorem allows us to conclude that most of these Salem numbers have degree equal to $4$. Indeed, the proof of Theorem~4(b) (loc.cit.) shows that the main contribution to the counting function $\mathcal{N}_r(\ell)$ comes from the ellipses $e(y)$ with $y \le x$, $y \not\in \Q$ and sums up to $\frac14\mathcal{E}(x)$, $x = 2\cosh \ell$ different lengths (cf. \cite[Section~4]{Marklof} for the definition of $\mathcal{E}(x)$). 
The corresponding Salem numbers $\lambda$ are square-rootable, hence $\deg(\lambda) = \deg(\lambda ^{\frac12})$  (by \cite[Lemmas~7.4 and 7.2]{EmeryRatcliffeTschantz}). On the other hand, we have $\lambda ^{\frac12} + \lambda ^{-\frac12} = y \not\in \Q$, hence $\deg(\lambda) >2$. 

This completes the proof of part A of the theorem.

\subsection{Proof of part B} \label{sec:proof B}

Here, we are interested in the case where $L=\Q$ and $\deg(\lambda)=4$. We begin by recalling the following lemma of Emery, Ratcliffe, and Tschantz:
\begin{lem}[\cite{EmeryRatcliffeTschantz}, Lemma 8.2 (2)] \label{lem:ERT 8.2}
 Let $\lambda$ be a Salem number of degree 4 and $p(x)$ its minimal polynomial. Then $\lambda$ is square-rootable over $\Q$ if and only if $p(-1)$ is a square in $\Z$. 
\end{lem}

Our goal is to count those $\lambda$'s of degree 4 that are square-rootable. 
Thus, we wish to count polynomials  
\[p(x)=x^4+ax^3+bx^2+ax+1 \in \Z[x]\]
such that 
\begin{itemize}
\item[(a)] $p(x)$ is irreducible;
\item[(b)] $p(x)$ is a Salem polynomial, that is, its roots are $\lambda, \lambda^{-1}, \mu, \mu^{-1}$ with $\lambda \in \R_{>1}$ and $\mu \not \in \R, |\mu|=1$;
\item[(c)] $p(-1)=k^2$ for $k\in \Z$;
\item[(d)] $\lambda \leq Q$. 
\end{itemize}
We remark that $k$ above must be different from $0$, since $-1$ cannot be a root of $p(x)$. Thus we can assume that $k>0$. 

Condition (c) is equivalent to 
\begin{equation}\label{eq:p(-1)}
2+b-2a=k^2, \mbox{ for } k\in\Z, k>0.
\end{equation}
Now we focus on condition (b). Our first observation is that 
\[-a=\lambda+\lambda^{-1}+\mu+\mu^{-1}< \lambda +3\leq Q+3.\]
In addition, $\lambda+\lambda^{-1}\geq 2$ implies that 
\[-a=\lambda+\lambda^{-1}+\mu+\mu^{-1}> 0.\]
Therefore, 
\begin{equation}\label{eq:a}
0< -a < Q+3.
\end{equation}

Write $y=x+x^{-1}$ and consider the polynomial 
\[r(y)=y^2+ay+b-2.\]
Then it is immediate to see that $x^2r(x+x^{-1})=p(x)$. 
Condition (b) is equivalent to asking that $r(y)$ has two real roots, with one $>2$ and the other in the interval $(-2,2)$. Writing the roots as 
\[\frac{-a\pm \sqrt{a^2-4(b-2)}}{2},\]
we see that we need 
\begin{equation}\label{eq:roots}
 a^2>4b-8.
\end{equation}
Combining with condition \eqref{eq:p(-1)}, we have 
\begin{equation}\label{eq:roots2}
(a-4)^2>4k^2.
\end{equation}
 In addition, we can rewrite
\begin{eqnarray*}
 2<\frac{-a+\sqrt{a^2-4(b-2)}}{2}& \mathrm{as} & 4+a <\sqrt{a^2-4(b-2)}, \ \mathrm{and} \\
2>\frac{-a-\sqrt{a^2-4(b-2)}}{2}>-2& \mathrm{as} & 4-a >\sqrt{a^2-4(b-2)}>-4-a.
 \end{eqnarray*}
Combining the above lower bounds for the square-root, we have 
\begin{equation}\label{eq:ab}
a^2-4(b-2)>(4+a)^2, \ \mathrm{which \ simplifies \ to} \ -2a>2+b.
\end{equation}
Similarly, the upper bound gives us 
\[(4-a)^2>a^2-4(b-2), \ \mathrm{which \ simplifies \ to} \ 2+b >2a,\]
but this condition is already a consequence of \eqref{eq:p(-1)}. 

Combining \eqref{eq:p(-1)} with \eqref{eq:ab} we obtain 
\begin{equation}\label{eq:ka}
k^2<-4a.
\end{equation}
Notice that $-16a<(a-4)^2$, and therefore equations \eqref{eq:roots} and \eqref{eq:roots2} are consequences of \eqref{eq:ka}.

In sum, we have the following conditions
\[b=k^2+2a-2, \quad k^2<-4a, \quad 0<-a<Q+3.\]
The number of solutions for this is given by
\[\sum_{j=1}^{Q+2} (\lceil\sqrt{4j}\rceil-1) =2\int_1^{Q+2}\sqrt{x} dx +O(Q)=\frac{4}{3}Q^{3/2}+O(Q).\]
We have not yet taken into account condition (a). The only way for $p(x)$ to be reducible is to have
\[p(x)=(x^2+\alpha x+1)(x^2+\beta x+1)\]
in $\Z[x]$, where one of the factors (say, the first) is the minimal polynomial of $\mu$. Since $|\mu|=1$, we conclude that $|\alpha|<2$ and therefore the only possible values for $\alpha$ are $0,\pm 1$. Choosing the value of $\alpha$ and comparing the conditions 
on the coefficients $a,b$ of $p(x)$, we have
\begin{eqnarray*}
 \alpha=0 & \mathrm{and} & b=2;\\
 \alpha=1& \mathrm{and} & b=a+1;\\
\alpha=-1 & \mathrm{and} & 1=a+b.
 \end{eqnarray*}
There are $O(Q)$ choices of $a$ and $b$ satisfying the three equations above, which completes the proof of part~B. \qed

\subsection{}
We remark that we can simplify our reasoning from the previous section to recover the result of G\"otze and Gusakova for the case where $m=1$ (which corresponds to Salem numbers of degree $4$).  Indeed, we must again count the possible polynomials $p(x)$, this time without condition (c). From the previous discussion, we have the conditions 
\[ -2a>2+b>2a, \quad 0<-a<Q+3.\]
We also have $a^2>4b-8$, but it is easy to see that this condition is a consequence of the first equation above. 

The number of solutions for the above inequalities is 
\[\sum_{j=1}^{Q+2} (4j-1)=2Q^2+O(Q),\]
and this recovers Theorem \ref{thm:GG} for $m = 1$.

\section{Cocompact case}

In this section we consider a generalization of Theorem~\ref{thm1} to compact orbifolds. Here the results are less precise and conditional on the \emph{gap conjecture} of Marklof:

\begin{conj}[\cite{Marklof}, Conjecture 1]\label{conj-gap}
Let $\Gamma_\mathfrak{D}$ be an arithmetic quaternion group. Then the number of gaps in the complex length spectrum of $\HH^3/\Gamma_\mathfrak{D}$ up to length $\ell = \log x$ is given by 
\[ \mathcal{G}(x)  = \kappa x + o(x), \quad x\to\infty, \]
where $\kappa \ge 0$ is a constant depending only on $\Gamma_\mathfrak{D}$, and small compared to $2^{2d-3}\pi |D_\mathfrak{a}|^{-1/2}$. It could even be the case that $\kappa = 0$ for all $\Gamma_\mathfrak{D}$. (Here $d$ denotes the degree of the field of definition $K$ and $D_\mathfrak{a}$ is the discriminant of $\mathfrak{a} = \tr \mathfrak{D}$.)
\end{conj}

The conjecture is known to be true (with $\kappa = 0$) when $\Gamma_\mathfrak{D} = \mathrm{PSL}(2, \mathfrak{o}_K)$ is a Bianchi group, but it remains open even for the other arithmetic groups of the simplest type. The main result of this section is the following theorem. 

\begin{thm}\label{thm2}
\begin{itemize}
\item[A.] Let $\Gamma_\mathfrak{D}$ be an arithmetic quaternion group of the simplest type with the totally real field of definition $L$, and let $\mathcal{O}_\mathfrak{D} = \HH^3/\Gamma_\mathfrak{D}$. Assume that Conjecture~\ref{conj-gap} holds for $\Gamma_\mathfrak{D}$. Then $\mathcal{O}_\mathfrak{D}$ generates
\[c_1 Q^{1/2} + o(Q^{1/2}), \quad Q\to\infty \]
 Salem numbers of degree $4$ over $L$ that are square-rootable over $L$, where $c_1 = \frac{2^{2d-3}\pi}{|D_\mathfrak{a}|} - \frac{\kappa}{4}$ with the notation as in Conjecture~\ref{conj-gap}.
\item[B.] Let $L$ be a totally real number field. There exists a constant $c_2 = c_2(L) \ge 0$ such that the number of Salem numbers of degree $4$ over $L$ that are square-rootable over $L$ and less than or equal to $Q$ is 
\[c_2 Q^{3/2}+o(Q^{3/2}).\]
\end{itemize}

\end{thm}
\begin{proof} \textbf{A.}
This part of the proof is similar to the argument in Section~\ref{sec:proof A}. The only difference is that instead of Marklof's Theorem~4(b) cited there we now apply his Theorem~4(a) together with the gap conjecture \cite{Marklof}.  Note that in Theorem~4(a) [loc. cit.] there is an additional assumption that the set of traces $\tr \mathfrak{D}^1$ is invariant under complex conjugation. This assumption allows us to prove Lemma~2 [loc. cit.] which is then used in the proof of the theorem. It is well known that arithmetic groups of the simplest type always have a totally real index two subfield of their complex field of definition (see \cite[Section~10.2]{MR-book}). Therefore, we are not required to impose the aforementioned extra assumption on the traces. The rest of the argument is the same as in Section \ref{sec:proof A}.   

\medskip

\noindent\textbf{B.} We now use the general form of square-rootable Salem numbers given in Definition~\ref{defn-sr}. We work over an arbitrary totally real field $L$, but consider only the Salem numbers $\lambda$ with $\deg_L(\lambda) = 4$. The latter assumption allows us to apply the method from Section~\ref{sec:proof B}. 

First recall a more general lemma from \cite{EmeryRatcliffeTschantz} (compare with Lemma~\ref{lem:ERT 8.2}):

\begin{lem}[\cite{EmeryRatcliffeTschantz}, Lemma 8.2 (1)]
 Let $\lambda$ be a Salem number, let $L\subset\mathbb{Q}(\lambda+\lambda^{-1})$ a subfield, and let $p(x)$ be the minimal polynomial of $\lambda$ over $L$. If 
 $p(x)=x^4+ax^3+bx^2+ax+1 \in \mathfrak{o}_{L}[x]$, then $\lambda$ is square-rootable over $L$ if and only if there is a positive element $k$ of $L$ such that $p(-1) = k^2$ and $4-a\pm 2k$
is a totally positive element of $L$, in which case $\lambda$ is square-rootable over $L$ via $4-a\pm 2k$. 
 \end{lem}

As before, we seek to count the Salem numbers $\lambda$ of degree $4$ over $L$ that are square-rootable. 
This amounts to counting the polynomials  
\[p(x)=x^4+ax^3+bx^2+ax+1 \in \mathfrak{o}_{L}[x]\]
for which
\begin{itemize}
 \item[(a)] $p(x)$ is irreducible;
\item[(b)] its root $\lambda$ is a Salem number of degree $4[L:\Q]$, that is, the roots are $\lambda, \lambda^{-1}, \mu, \mu^{-1}$ with $\lambda \in \R_{>1}$ and $\mu \not \in \R, |\mu|=1$, and for all non-identity places $\sigma: L\to\R$ the roots of $p^{\sigma}(x)$ have absolute value one (cf. \cite[Proof of Theorem~5.2(1)]{EmeryRatcliffeTschantz});
\item[(c)] $p(-1)=k^2$ for $k\in \mathfrak{o}_{L}$, $k > 0$;
\item[(d)] $4-a+2k$ or $4-a-2k$ is totally positive; 
\item[(e)] $\lambda \leq Q$. 
\end{itemize}

The proof proceeds much as in Section \ref{sec:proof B}. Here, condition (c) is equivalent to 
\begin{equation}\label{eq:p(-1)-c}
2+b-2a=k^2, \mbox{ for } k\in\mathfrak{o}_{L}, k>0.
\end{equation}
Now we turn our attention to condition (b), which allows us to deduce that
\[-a=\lambda+\lambda^{-1}+\mu+\mu^{-1}< \lambda +3\leq Q+3,\]
where the final inequality follows from condition (e).
Moreover, since $\lambda$ is a Salem number, we know that $\lambda+\lambda^{-1}\geq 2$, hence
\[-a=\lambda+\lambda^{-1}+\mu+\mu^{-1}> 0.\]
Combining the two displayed inequalities for $-a$ yields
\begin{eqnarray}\label{eq:a-c}
 0   <  -a < Q+3. 
\end{eqnarray}
Furthermore, since $p^\sigma(x)$ has all of the roots with absolute value one, it must be the case that $|a^\sigma| < 4$ for any non-identity $\sigma: L\to\R$.

Next, we perform a change of variable, writing $y=x+x^{-1}$. Then, if we define
\[r(y)=y^2+ay+b-2,\]
we see that $p(x)$ can be expressed in terms of this new polynomial: $x^2r(x+x^{-1})=p(x)$. 
In other words, condition (b) amounts to requiring that $r(y)$ has two real roots, with one $>2$ and the other in the interval $(-2,2)$, whose $\sigma$-conjugates are all in the interval $(-2,2)$. The roots of $r(y)$ are of the form
\[\frac{-a\pm \sqrt{a^2-4(b-2)}}{2}.\]
As a result, we need 
\begin{equation}\label{eq:roots-c}
 (a^\sigma)^2>4b^\sigma - 8 \mbox{ for all } \sigma: L\to\R.
\end{equation}
Combining this with condition \eqref{eq:p(-1)-c} yields
\begin{equation}\label{eq:roots2-c}
(a^\sigma - 4)^2 > 4(k^\sigma)^2.
\end{equation} 
Hence, from \eqref{eq:a-c}, for all non-identity $\sigma: L\to\R$ we have
\begin{equation}\label{eq:k^s-c}
-4 < k^\sigma <4.
\end{equation} 
We need to impose some additional assumptions on $k^\sigma$ in order to satisfy condition~(d). More precisely, we have that either
\[\frac{a^\sigma -4}{2} < k^\sigma <4\]
for  all non-identity $\sigma: L\to\R$
or that
\[-4<k^\sigma <\frac{4-a^\sigma}{2}\]
for all non-identity $\sigma: L\to\R$. In both cases, this condition replaces \eqref{eq:k^s-c}.


Furthermore, taking into account the roots of $r(y)$ and the intervals that they live in, we deduce the following inequalities: 
\begin{eqnarray*}
 2<\frac{-a+\sqrt{a^2-4(b-2)}}{2}& \mathrm{i.e.,} & 4+a <\sqrt{a^2-4(b-2)};\\
2>\frac{-a-\sqrt{a^2-4(b-2)}}{2}>-2& \mathrm{i.e.,} &  4-a >\sqrt{a^2-4(b-2)}>-4-a; \mbox{ and }\\
\mbox{ for all non-identity }\sigma: L\to\R  &\mbox{have}& 4 \pm a^\sigma >\sqrt{(a^\sigma)^2 - 4(b^\sigma - 2)}> -4 \pm a^\sigma.
 \end{eqnarray*}
Next, we combine the lower bounds for the square-root that we obtained above, which yields 
\begin{equation}\label{eq:ab-c}
a^2-4(b-2)>(4+a)^2, \ \mathrm{which \ simplifies \ to} \  -2a>2+b.
\end{equation}
Likewise, we combine the upper bounds, which gives us 
\[(4-a)^2>a^2-4(b-2) , \ \mathrm{which \ simplifies \ to} \ 2+b >2a.\]
Note that this inequality is already a consequence of \eqref{eq:p(-1)-c}. 

For the conjugates, taking into account \eqref{eq:a-c}, we only have a non-trivial upper bound for the square root, which gives
\[ 2+b^\sigma > \pm 2a^\sigma.\]

Next, we combine \eqref{eq:p(-1)-c} with \eqref{eq:ab-c}, which produces a simple inequality 
\begin{equation}\label{eq:ka-c}
k^2<-4a.
\end{equation}
Observe that $-16a<(a-4)^2$, which means that equations \eqref{eq:roots-c} and \eqref{eq:roots2-c} are consequences of \eqref{eq:ka-c}.

To summarize, we have shown that the following inequalities must simultaneously hold:
\[b=k^2+2a-2;\] 
\[k^2<-4a;\] 
\begin{equation}\label{eq:diagonal}\frac{a^\sigma -4}{2} < k^\sigma <4 \mbox{ or } -4<k^\sigma <\frac{4-a^\sigma}{2};\end{equation}
\[0<-a<Q+3, \quad -4 <a^\sigma < 4.\]
Condition \eqref{eq:diagonal} should be interpreted as $\frac{a^\sigma -4}{2} < k^\sigma <4 $ for all non-identity $\sigma$ or $-4<k^\sigma <\frac{4-a^\sigma}{2}$ for all non-identity $\sigma$.

The number of solutions can be counted in a manner similar to \cite[p. 525]{Marklof}, following standard methods from the geometry of numbers (see, for example, \cite[Theorem 1, Chapter V]{Lang}). The count that we obtain is of the form $c_2 Q^{3/2}+o(Q^{3/2})$, where $c_2$ is a nonnegative constant. We include some details of this computation in the next lemma.
 

Observe that we still have not used condition (a) in our count. As in the previous section, the only way for $p(x)$ to be reducible is for it to factor into a product of quadratics, i.e.,  
\[p(x)=(x^2+\alpha x+1)(x^2+\beta x+1)\]
in $\mathfrak{o}_{L}[x]$, with one of them being the minimal polynomial of $\mu$. Without loss of generality, suppose that the first factor has this property. Since its roots have absolute value $1$ for all $\sigma: L\to \R$, we conclude that $|\alpha^\sigma|<2$ for all $\sigma$. This gives finitely many choices for $\alpha \in \mathfrak{o}_L$, and hence $O(Q)$ choices for $p(x)$.
\end{proof}

In fact, we can use the geometry of numbers to give a bound for $c_2$. Let $h=[L:\Q]$. Then we have $h$ embeddings of $L$ into $\R$ given by $\sigma_1=1, \sigma_2,\dots,\sigma_h$. Consider the function 
\[\varphi: L \rightarrow \R^h\]
\[\varphi(x)=(x^{\sigma_1}, x^{\sigma_2},\dots, x^{\sigma_h}).\]
It is well-known that the image of $\mathfrak{o}_{L}$ is a full lattice whose fundamental domain $\phi_L$ is a parallelotope of volume $|D_L|^{1/2}$. 

\begin{lem}
Let $c_2$ be as in Theorem \ref{thm2}.B. Then 
\[c_2\leq \frac{2^{2h+2}(12+7\delta+\delta^2)^{h-1}}{3|D_L|},\]
where \[\delta=2\min_{\phi_L} \max_{\mathrm{diagonal}} \phi_L.\]
In other words, $\delta$ is twice the minimal value of the maximal diagonals of all the possible parallelotopes $\phi_L$ corresponding to fundamental domains of the lattice given by $\mathfrak{o}_L$. 
\end{lem}
\begin{proof}
 We concentrate on counting the number of $a,k \in \mathfrak{o}_{L}$ such that 
 \begin{equation}\label{eq:system}
 \begin{cases}
 0<-a<Q+3,\\
-4 <a^\sigma < 4 & \forall \sigma \not = 1, \\
 k^2<-4a, \\
 \frac{a^\sigma -4}{2} < k^\sigma <4& \forall \sigma \not = 1.\\
\end{cases} 
 \end{equation}
The result with $-4<k^\sigma <\frac{4-a^\sigma}{2}$ is analogous and will yield the same number.
 
Now we consider two coordinates, namely $\tilde{\varphi}:L\times L  \rightarrow \R^{2h}$ given by $\tilde{\varphi}(a,k)=(\varphi(a),\varphi(k))$.  The image of $\mathfrak{o}_{L}\times \mathfrak{o}_{L}$ is a full lattice in $\R^{2h}$ whose fundamental parallelotope $\tilde{\phi_L}$ has volume $|D_L|$.

The number of solutions to \eqref{eq:system} is approximated by the number of translates of $\tilde{\phi_L}$ by the image of $\mathfrak{o}_{L}\times \mathfrak{o}_{L}$ that fit in the set 
$S_L(Q,0)$, where 
\begin{align*}
S_L(Q,\delta)=\Big\{({\bf x},{\bf y})\in  \R^{2h}\, \Big|\,& -\delta<-x_1<Q+3+\delta, |x_i|<4+\delta,  |y_1|<\sqrt{-4x_1}+\delta,\\
&\frac{x_i -4}{2}-\delta < y_i <4+\delta, \,  i=2,\dots,h\Big\}.
\end{align*}
More precisely, let $n(Q,\delta)$ be the number of translations of $\tilde{\phi_L}$ which are contained in  $S_L(Q,\delta)$, and let $m(Q,\delta)$ be the number of translations of $\tilde{\phi_L}$ which intersect  $S_L(Q,\delta)$. Let $\ell(Q)$ be the number of lattice points in $S_L(Q,0)$. Then we have 
\begin{equation}\label{eq:vol}
n(Q,\delta)\mathrm{Vol}(\tilde{\phi_L})\leq \mathrm{Vol}(S(Q,\delta)) \leq m(Q,\delta)\mathrm{Vol}(\tilde{\phi_L})
\end{equation}
and
\[n(Q,0)\leq \ell (Q)\leq m(Q,0).\]
Let $\delta$ be equal to the length of the longest diagonal of $\tilde{\phi_L}$. Thus $m(Q,0)\leq n(Q,\delta)$. This yields
\begin{equation}\label{eq:bounds}
n(Q,0)\leq \ell(Q)\leq m(Q,0)\leq n(Q,\delta).
\end{equation}
In textbook applications, one also normally writes $m(Q,-\delta)\leq n(Q,0)$. However, we are not able to do this here because our set $S_L(Q,\delta)$ is thin compared with the size of $\delta$ in the directions where $\sigma\not = 1$. This is the reason why we get an upper bound, but no lower bound.  

Combining equations \eqref{eq:vol} and \eqref{eq:bounds}, we obtain
\begin{equation}\label{eq:lbound}
\ell(Q)\leq \frac{\mathrm{Vol}(S(Q,\delta)) }{\mathrm{Vol}(\tilde{\phi_L})}.
\end{equation}
Notice that 
\begin{align*}
\mathrm{Vol}(S_L(Q,\delta))=&\int_{-Q-3-\delta}^\delta \int_{-4-\delta}^{4+\delta}\cdots \int_{-4-\delta}^{4+\delta} 2(\sqrt{-4x_1}+\delta)  \prod_{i=2}^h \left(6-\frac{x_i}{2}+2\delta\right) dx_1\dots dx_h\\=&(48+28\delta+4\delta^2)^{h-1}\frac{8}{3} Q^{3/2} + O_\delta(Q).
\end{align*}
Combining with \eqref{eq:lbound}, we arrive at the claimed expression for $\delta$, which is the maximal diagonal of $\tilde{\phi_L}$ and can therefore be bounded by twice the maximal diagonal of $\phi_L$. 


\end{proof}

\section{Comments about other dimensions}

\subsection{} We can consider the two dimensional case using the previous work of Bolte \cite{Bolte} instead of Marklof. Let $\mathcal{O}$ be a non-compact arithmetic $2$-orbifold with associated group $\Gamma$. By \cite[Theorem~1.1]{EmeryRatcliffeTschantz}, for every hyperbolic element $\gamma\in \Gamma$ with $\lambda = e^{\ell(\gamma)}$, we have that $\lambda$ is a Salem number of degree $2$ (recall that the Salem numbers have even degree). It follows that a non-compact arithmetic hyperbolic $2$-orbifold generates $Q^{1/2}+O(1)$ different Salem numbers of degree $2$ that are less than or equal to $Q$.

On the arithmetic side, the counting here is very simple. For degree $2$, we have $m=0$, and the problem of counting Salem number reduces to counting irreducible polynomials of the form $x^2+ax+1 \in \Z[x]$ under the condition that $0<-a<Q+1$. It is also easy to see that the case $a=-1$ does not yield real roots, while $a=-2$ gives $\lambda=1$. It follows that the number of Salem numbers of degree $2$ less than or equal to $Q$ is $Q-2$.

So, in the logarithmic scale, a given non-compact arithmetic $2$-orbifold $\mathcal{O}$ generates asymptotically $1/2$ of all Salem numbers of degree $2$. 

\subsection{} Now consider an arbitrary dimension $n > 1$. By the prime geodesic theorem of Margulis \cite{Margulis} the number of geodesics of a compact hyperbolic $n$-manifold of length at most $\ell$ grows like $e^{(n-1)\ell}/(n-1)\ell$. If the geodesics all had distinct lengths, then in dimensions $2$ and $3$ we would have about $cQ$ different Salem numbers defined by a single arithmetic manifold $\mathcal{M}$. This is a much larger number than what we expect to obtain from Theorem~\ref{thm2}. The issue here is not in the compactness of $\mathcal{M}$ but rather it stems from the fact that the geodesic spectra of arithmetic manifolds tend to have large multiplicities. 

Determining the multiplicity of geodesics with a given length is known to be a very difficult problem. For example, Sarnak \cite{Sar82} studied this problem for the modular surfaces $\HH^2/ \Gamma$, where $\Gamma$ is a congruence subgroup of $\mathrm{PSL}(2,\Z)$,  in which case the multiplicities are the class numbers of indefinite binary quadratic forms. He used Selberg's trace formula to determine the asymptotic growth of their average sizes. Subsequent papers by Bolte and Marklof cited above treat the mean multiplicities in the spectra of $2$ and $3$-dimensional arithmetic orbifolds. Very little is known about multiplicities in higher dimensions. The relation between the geodesic spectrum and the distribution of Salem numbers that we investigate in this paper allows us to obtain many new multiplicity bounds. For example, we can now can prove Proposition~\ref{prop-mult}, which we restate below.

{
\renewcommand{\thethm}{\ref{prop-mult}}

\begin{prop}
Let $\mathcal{O}$ be a non-compact arithmetic hyperbolic orbifold of even dimension $n \ge 4$. Then the mean multiplicities in the length spectrum of $\mathcal{O}$ have a strong exponential growth rate of at least 
$$ \langle g(\ell) \rangle \sim c\frac{e^{(\frac{n}2-1)\ell}}{\ell}, \quad \ell\to\infty,$$
where $c$ is a positive constant.

\end{prop}

\addtocounter{thm}{-1}

}

\begin{proof}
The number of closed geodesics of $\mathcal{O}$ of length at most $\ell$ is $\sim e^{(n-1)\ell}/(n-1)\ell$. Note that $\mathcal{O}$ is not compact and has singularities so we cannot apply Margulis' theorem; we refer instead to \cite[Porposition~5.4]{GanWar} where the result is obtained in this setting. Now by \cite[Theorem~1.1]{EmeryRatcliffeTschantz} each geodesic corresponds to a Salem number $e^\ell$ of degree $\le n$ (here we use that $n$ is even and that the degrees of the Salem numbers are even). By \cite[Theorem~1.1]{GoetzeGusakova} the total number of such Salem numbers is bounded by $c e^{(n/2)\ell}$. Hence, on average, the geodesic lengths have to appear with multiplicity at least  $\sim e^{(n/2-1)\ell}/(n-1)\ell$.
\end{proof}

Extending this result to compact orbifolds and to odd dimensions requires finer counting of square-rootable Salem numbers and \emph{relative Salem numbers}, i.e., those $\lambda$ for which $\Q(\lambda+\lambda^{-1})$ contains a fixed field $L$. We leave these intriguing problems for future research.  It would also be interesting to find the proportion of Salem numbers defined by a given arithmetic manifold or orbifold. At this point, we cannot discount the possibility that in large dimensions the exponent of the growth function of such Salem numbers is the same as the exponent in the total growth of the admissible Salem numbers.

\subsection*{Acknowledgements.} The authors would like to thank Chris Smyth for helpful comments on an earlier draft of this manuscript. This project was started at the BIRS-CMO workshop ``Number Theory in the Americas'' in Oaxaca, Mexico. We sincerely thank the organizers of the workshop, and the Casa Matem\'{a}tica Oaxaca, for making it possible for us to work together. 

MB is partially supported by the CNPq, FAPERJ and by the Max Planck Institute for Mathematics.  ML is partially supported by the Natural Sciences and Engineering Research Council of Canada (NSERC) Discovery Grant 355412-2013 and by the Fonds de recherche du Qu\' ebec -- Nature et technologies (FRQNT) Projet de recherche en \'equipe 256442. PM was supported by the KIAS Individual Grant MG072601 at the Korea Institute for Advanced Study. LT was supported by the Max Planck Institute for Mathematics during her sabbatical in the 2019-2020 academic year, and by Oberlin College during the early stages of this paper. 

\bibliographystyle{amsalpha}

\bibliography{Bibliography}

\end{document}